\documentclass[12pt]{article}
\usepackage{tikz}
\usepackage{subfigure}
\usepackage[english]{babel}
\usepackage{amsfonts,amssymb,amsmath,latexsym,amsthm}
\usepackage{multirow}

\newtheorem{thm}{Theorem}[section]
\newtheorem{cor}[thm]{Corollary}

\newtheorem{prob}[thm]{Problem}
\newtheorem{claim}{Claim}
\newtheorem{fact}{Fact}

\begin{document}
\title{The score sequences with unique tournament that has minimum number of upsets
\thanks{The work was supported by NNSF of China (No. 11671376) and Anhui Initiative in Quantum Information Technologies (AHY150200).}}
\author{Yuming Zhang$^a$,\quad Xinmin Hou$^b$\\
\small $^{a,b}$ Key Laboratory of Wu Wen-Tsun Mathematics\\
%\small Chinese Academy of Sciences\\
\small School of Mathematical Sciences\\
\small University of Science and Technology of China\\
\small Hefei, Anhui 230026, China.\\
}
\date{}
\maketitle

\begin{abstract}
Let $T$ be a tournament with nondecreasing score sequence $R$ and $A$ be its tournament matrix.
An upset of $T$ corresponds to an entry above the main diagonal of $A$. Given a feasible score sequence $R$, Fulkerson~(1965) gave a simple recursive construction for a tournament with score sequence $R$ and the minimum number of upsets,  and Hacioglu et al. (2019) provided a construction for all of such tournament matrices.  Let $U_{\min}(R)$ denote the set of tournament matrices with score sequence $R$ that have minimum number of upsets. Brauldi and Li~(1983) characterized the strong score sequences $R$ ($R$ is strong if a tournament $T$ with score sequence $R$ is strongly connected) with $|U_{\min}(R)|=1$. In this article, we characterize all feasible score sequences $R$ with $|U_{\min}(R)|=1$ and give an explicit formula for the number of the feasible score sequences $R$ with $|U_{\min}(R)|=1$.
%get the generalized form of the score sequences satisfying Landau's condition that the tournament with the minimum number of upsets whose score-list with that form is unique. We also caculate the numbers of the score sequences satisfying above property for fixed order, we give a recursion formula about the order of the score sequences.\\

\par\textbf{Keywords: }tournament;upsets;score sequences
\end{abstract}

\section{Introduction}
A {\em tournament} is an orientation graph of a complete graph.
%in which every pair of distinct vertices is connected by a direct edge with any one of two possible orientations.
The {\em score-list} of a tournament is the sequence of the outdegrees of its vertices.
A {\em tournament matrix } is  the adjacency matrix, $A=\left(a_{ij}\right)$, of a tournament. Then $A$ is a (0,1)-matrix satisfying that $A+A^T=J-I$, where $J$ is the all 1's matrix and $I$ is the identity matrix.
So the row sum vector of $A$ is the score-list of the tournament, also called the {\em score sequence} of $A$.
Throughout the paper, we identify $(0,1)$-matrices and digraphs if no confusion from the context.
%$a_{ij}=1$ if and only if $a_{ji}=0$, and $a_{ii}=0$, for $i=1,2,\dots,n$.

Let $\mathcal{T}(R)$ be the set of all $n$-by-$n$ tournament matrices with row sum vector $R=(r_1, r_2,\ldots, r_n)$. For a fixed nondecreasing integral vector $R=\left(r_1,r_2,\dots,r_n\right)$ with
\begin{equation}\label{EQ: e1}
\sum\limits_{i=1}^nr_i=\binom{n}{2},
\end{equation} the Landau inequalities~\cite{Landau53}
\begin{equation*}
\sum_{i=1}^kr_i\ge\binom{k}{2},\ \mbox{ for } k=1,2,\ldots,n-1 \label{Landau}
\end{equation*}
provide sufficient and necessary conditions for the existence of a tournament matrix in $\mathcal{T}(R)$.
We call a nondecreasing integral vector $R=(r_1, r_2, \ldots, r_n)$ a {\em feasible score sequence} if $R$ satisfying (\ref{EQ: e1}) and the Landau inequalities.
One of the fundamental problems in the study of the tournaments was generating all the tournament matrices in $\mathcal{T}(R)$, which has been solved
independently by Kannan et al.~\cite{KPV-RSA99} and McShine~\cite{McShine-EJC00}.
%Throughout this paper, a score sequence always means a feasible score  unless otherwise stated.

Let $T$ be a tournament with feasible score-list $R$ and $A=(a_{ij})$ be its tournament matrix.
An {\em upset} of $T$ corresponds to an entry $a_{ij}=1$ with $i<j$ in $A$. In fact, regard $n$ vertices of the tournament $T$ as $n$ teams, note that the score sequence is nondecreasing,  an upset is a game that the team $j$ lost to the team $i$, but $j>i$. So for the team with the higher score it is an upset game.
% that's why we call it an upset.
%\indent Obviously the minimum number of upsets is at least $\ell=\sum_{i=1}^nh_i$, where $h_i$ stands for the $\max\{0,r_i-\left(i-1\right)\}$. In 1964, Ryser \cite{Ryser64} proved that $\ell$ is actually the minimum number of upsets. Then Fulkerson~\cite{Fulkerson65} provided a construction of a tournament with the minimum number of upsets.\\

For a feasible  score sequence $R=(r_1,r_2,\ldots,r_n)$, let $H_R=(h_1,h_2,\ldots,h_n)$, where $h_i=r_i-(i-1)$ for $i=1,2,\ldots,n$. $H_R$ is called the {\em normalized score vector} of $R$.
% and if $h_i$ is positive we denote it by $h_i^+$, if it is negative denote it by $h_i^-$.
Define multisets $X=\{h_i\cdot i : h_i>0 \}$ and $Y=\{ (-h_j)\cdot j : h_j<0\}$, where $h_i\cdot i$ stands for $h_i$ $i$'s.
Let $X'=\{i : h_i>0 \}$ and $Y'=\{j : h_j<0\}$. Then, by (\ref{EQ: e1}), we have $\sum\limits_{i\in X'}h_i=-\sum\limits_{j\in Y'}h_j$. Obviously the minimum number of upsets is at least $\sum\limits_{i\in X'}h_i$. In fact, Ryser~\cite{Ryser64} proved that the lower bound is actually the minimum number of upsets. Fulkerson~\cite{Fulkerson65} provided an algorithm for the construction of a tournament with the minimum number of upsets. Denote $\ell=\sum\limits_{i\in X'}h_i$.  A {\em feasible} $\ell$-tuple of $X\times Y$ is $\ell$ distinct ordered pairs $(i,j)\in X\times Y$ with $i<j$, where $i\in X$ occurs $h_i$ times and $j\in Y$ occurs $-h_j$ times in the ordered pairs.
Let $P_R$ be the set of all feasible $\ell$-tuples of $X\times Y$ and let $U_{\min}(R)$ denote the set of  tournament matrices  with score sequence $R$ that have minimum number of upsets.
%Now we form a set $P_R$ which consists of all $\ell$-tuples of the distinct ordered pairs $(i,j)$ where $i\in X,\ j\in Y$ and $i<j$, each $i$ occurs $h_i$ times and each $j$ occurs $-h_j$ times in the ordered pairs.\\
In~\cite{HSKF19}, Hacioglu et al. showed that

\begin{thm}[Theorem 2.1 in~\cite{HSKF19}]\label{thm1}
Let $R=(r_1,r_2,\dots,r_n)$ be a feasible score sequence. Then $|U_{\min}(R)|=|P_R|$ and the ordered pairs in each
feasible $\ell$-tuple give us the location of the upsets..
%Let $H=(h_1,h_2,\dots.h_n)$ be the normalized score vector of $R$ defined as before. Then the number of tournament matrices with row sum $R$ that have minimum number of upsets is equal to the cardinality of $P$ and the ordered pairs in each $\ell$-tuple give us the location of the upsets.
\end{thm}

Moreover, Hacioglu et al. also showed that

\begin{thm}[Theorem 5.1 in~\cite{HSKF19}]\label{THM: Regular}
Let $R_0=(\lfloor\frac{n-1}{2}\rfloor,\ldots, \lfloor\frac{n-1}{2}\rfloor,\lceil\frac{n-1}2\rceil, \ldots,\lceil\frac{n-1}{2}\rceil)$, where $\lfloor\frac{n-1}{2}\rfloor$ and $\lceil\frac{n-1}2\rceil$ each occurs $\frac n2$ times when $n$ is even (called regular when $n$ is odd and near-regular when $n$ is even) score sequences in~\cite{HSKF19}, respectively).
Then $|U_{\min}(R_0)|=1$.
\end{thm}

In fact, the above theorem can be viewed as a corollary of a result given by Brauldi and Li~\cite{Brualdi-Li83}.
A feasible score sequence $R$ is strong if a tournament $T$ with score sequence $R$ is strongly connected.
\begin{thm}[Theorem 2.7 in~\cite{Brualdi-Li83}]\label{THM: Brau-Li83}
 Let $R=(r_1, r_2, \ldots, r_n)$ be a strong score vector. Then $|U_{\min}(R)|=1$ if and
only if $$R=(\underbrace{k,\ldots, k}_k, k, k + 1 ,\ldots, n-k-1, \underbrace{n-k-1,\ldots,n-k-1}_k)$$
for some integer $k>1$ satisfying $2k + 1<n$.
\end{thm}

This motivates us to consider the following problems.
\begin{prob}\label{PROB: p1}
(I) Characterize all the feasible score sequences $R$ with $|U_{\min}(R)|=1$.

(II) How many feasible score sequences $R$ with the property that $|U_{\min}(R)|=1$.
\end{prob}

%In this note, we solve Problem~\ref{PROB: p1} completely.
%$R_0$ is called regular (when $n$ is odd) and near-regular (when $n$ is even) score sequences in~\cite{HSKF19}.
Note that  the normalized score vector of the strong score sequences is
$$H_R=(k, k-1, \ldots, 1, 0, \ldots, 0, {-1, -2,\ldots, -k}).$$
%Note that  the normalized score vector of the regular and near-regular score sequences are
%$$\left(\left\lfloor\frac{n-1}{2}\right\rfloor, \left\lfloor\frac{n-1}{2}\right\rfloor-1, \ldots, 1, 0, -1, \ldots, -\left\lfloor\frac{n-1}2\right\rfloor+1, -\left\lfloor\frac{n-1}{2}\right\rfloor\right)$$  and  $$\left(\left\lfloor\frac{n-1}{2}\right\rfloor, \left\lfloor\frac{n-1}{2}\right\rfloor-1, \ldots, 1, 0, 0, -1, \ldots, -\left\lfloor\frac{n-1}2\right\rfloor+1, -\left\lfloor\frac{n-1}{2}\right\rfloor\right),$$ respectively.
%We call a vector the same as the normalized score vectors of the regular and near-regular  score sequences a symmetric vector. Formally,
In general, we call  a vector $H$ {\it symmetric} if $H$ has the form $(0,\ldots,0,)$ or $(p,p-1,\ldots,1,0,\ldots,0,-1,-2,\ldots,-p)$ for some positive integer $p$. The following result generalizes Theorem~\ref{THM: Regular} and solves problem (I).

\begin{thm}\label{THM: main1}
Let $R=(r_1, r_2, \ldots, r_n)$ be a feasible score sequence. Then $|U_{\min}(R)|=1$
% the tournament matrix with row sum $R$ that have minimum number of upsets is unique
if and only if the normalized score vector of $R$ has the form
	\begin{equation}
	H_R=(H_1,H_2,\dots,H_m) \notag
	\end{equation}
	where every segment $H_i$ is a symmetric vector for some positive integer $p_i$, $i=1, 2, \ldots, m$.
	%\begin{equation}
	%H_i=(0,0,\dots,0,p_i,p_i-1,\dots,1,0,0,\dots,0,-1,-2,\dots,-p_i), \notag
	%\end{equation}
\end{thm}

The following theorem compute the number of the feasible score sequences $R$ with the property that $|U_{\min}(R)|=1$ and so answers problem (II).
%Now we  the number of the score sequences satisfying Theorem 2 for a fixed positive integer $n$.\\
\begin{thm}\label{THM: main2}
The number of the feasible score sequences $R$ of length $n$ with the property that $|U_{\min}(R)|=1$ is
\begin{multline*}
c_1\left(\frac{1-\sqrt{2\sqrt{5}+3}}2\right)^{n-1}+c_2\left(\frac{1+\sqrt{2\sqrt{5}+3}}2\right)^{n-1} \\
+c_3\left(\frac{1-i\sqrt{2\sqrt{5}-3}}2\right)^{n-1}+c_4\left(\frac{1+i\sqrt{2\sqrt{5}-3}}2\right)^{n-1},
 \end{multline*}
%$$c_1\left(\frac{1-\sqrt{2\sqrt{5}+3}}2\right)^{n-1}+c_2\left(\frac{1+\sqrt{2\sqrt{5}+3}}2\right)^{n-1}+c_3\left(\frac{1-i\sqrt{2\sqrt{5}-3}}2\right)^{n-1}
%+c_4\left(\frac{1+i\sqrt{2\sqrt{5}-3}}2\right)^{n-1},$$
where
\begin{align*}
  c_1=\frac{\sqrt{5}+1}4\left(\frac{1}{\sqrt{5}}-\frac{1}{\sqrt{2\sqrt{5}+3}}\right), & \, c_2=\frac{\sqrt{5}+1}4\left(\frac{1}{\sqrt{5}}+\frac{1}{\sqrt{2\sqrt{5}+3}}\right), \\
   c_3=\frac{\sqrt{5}-1}4\left(\frac{1}{\sqrt{5}}-\frac{i}{\sqrt{2\sqrt{5}-3}}\right), & \, c_4=\frac{\sqrt{5}-1}4\left(\frac{1}{\sqrt{5}}+\frac{i}{\sqrt{2\sqrt{5}-3}}\right).
\end{align*}

\end{thm}

The rest of the article is arranged as follows. We give the proof of Theorem~\ref{THM: main1} in Section 2, and in Section 3, we prove Theorem~\ref{THM: main2}.

\section{Proof of Theorem~\ref{THM: main1}}
From (\ref{EQ: e1}) and the Landau inequalities, we have the following fact.
\begin{fact}\label{FACT: Landau}
Let $R=(r_1,r_2,\ldots,r_n)$ be a feasible score sequence and  $H_R=(h_1,h_2,\ldots,h_n)$ be its normalized score sequence. Then $h_i-h_{i+1}\le 1$ for $i=1,2,\ldots,n-1$, $h_1\ge 0$, $h_n\le 0$, $\sum\limits_{i=1}^{k}{h_i}\ge 0$ for  $k=1,2,\dots,n-1$, and $\sum\limits_{i=1}^{n}{h_i}=0$.
\end{fact}

\begin{proof}[Proof of Theorem~\ref{THM: main1}]
	\emph{Sufficiency:}  Let $X$ and $Y$ be the multisets determined by the normalized score vector $H_R$. From Theorem~\ref{thm1}, it is sufficient to show that $|P_R|=1$, i. e. to show all the ordered pairs $(x,y)\in X\times Y$ in a feasible $\ell$-tuple are determined uniquely.
%If $h_i=0$, from Theorem 1, $i\notin X,i\notin Y$. So correspond it to the tournament matrix, the row $i$ has to be $(1,1,\dots,1,0,0,\dots,0)$, there are exactly $i-1$ ``1''s.\\
Let $M$ be a feasible $\ell$-tuple of $P_R$. Without loss of generality, assume $H_1=(h_{i_1},\dots,h_{i_p},h_{k_1},\ldots,h_{k_r},h_{j_1},\ldots,h_{j_p})$, where $h_{i_s}=p-s+1$, $h_{j_s}=-s$ for $s=1,2,\ldots, p$  and $h_{k_1}=\ldots=h_{k_r}=0$.
%	\indent Except $i_1,i_2,\dots,i_p$, choose any other $i\in X$, we have $j_1,j_2,\dots,j_p<i$.
Since $h_{j_p}=-p$, $j_p$ occurs in $p$ distinct ordered pairs of the form $(i,j_p)$ with $i\in X$ and $i<j_p$.
Since there are exactly $p$ distinct elements $i_1, \ldots, i_p$ less than $j_p$ in $X$, the $p$ ordered pairs in $M$ containing $j_p$ have to be $(i_1,j_p),\ldots,(i_p,j_p)$. Because $h_{i_p}=1$ and $i_p$ occurs in $(i_p,j_p)$, $(i_p,j_{p-1})$ can not belong in $M$.
And since $h_{j_{p-1}}=-(p-1)$, $j_{p-1}$ has to occur in $(i_1,j_{p-1}),\ldots,(i_{p-1},j_{p-1})$ with a same reason. Because $h_{i_{p-1}}=2$ and $(i_{p-1}, j_p), (i_{p-1}, j_{p-1})\in M$, $(i_{p-1}, j_{p-2})$ can not occur in $M$.
And since $h_{j_{p-2}}=-(p-2)$, $j_{p-2}$ has to occur in $(i_1,j_{p-2}),\ldots,(i_{p-2},j_{p-2})$. Continue this procedure, we have $j_1$ has to occur in $(i_1, j_1)$.  Therefore, all the ordered pairs in $M$ with entries being indices in $H_1$ are determined uniquely, i.e.  these ordered pairs are independent with the pairs with entries being indices out of $H_1$. So with similar discussion on the ordered pairs in $M$ with entries being indices in $H_2$, we have that all such ordered pairs in $M$ are determined uniquely too. Continue the same discussion on $H_3, \ldots, H_m$ one by one, we have all ordered pairs in $M$ are determined uniquely. This completes the proof of the sufficiency.
%{\color{red} Same with $j_{p-2},j_{p-3},\dots,j_1$.}\\
%	\indent Hence the rows with sum $H_1$ in the corresponding tournament matrix is uniquely identified. {\color{red} Same with $H_2,H_3,\dots,H_m$.} So the corresponding tournament matrix is unique.\\

\emph{Necessity:}
%If $R$ is a score sequence satisfying Landau condition, and the tournament matrix with row sum $R$ that have minimum number of upsets is unique, we assume that
Let $H$ be the normalized score vector of $R$ and $H_1$ be the first maximal segment starting with a positive string  and ending in a {nonpositive string with a negative end of $H$}, i.e.
$H_1=(p,\ast,\ldots, \ast, 1, 0,\ldots,0,-1, \ast, \ldots,\ast, -r)$, where the predecessor of $p$ is zero (if any) and there is no negative entries between  $-r$ and the second positive string (if any) by the maximality of $H_1$. Let $-q$ be the minimum entry of the nonpositive string $(-1, \ast, \ldots,\ast, -r)$ and
%where $-q$ is the minimum component of the first non-positive integer string in $H$, i.e.\ there is a string with the form $(h_{k_1},h_{k_2},\dots,h_{k_t})$, where $h_{k_i}\le 0,i=1,2,\dots,t,h_{k_1-1}>0,h_{k_t+1}>0$, and for $\forall i<k_1$, $h_i\ge 0$, $-q=\min\{h_{k_1},h_{k_2},\dots,h_{k_t}\}$.\\
%	\indent First, the Landau condition imply that $h_i-h_{i+1}\le 1$, for $i=1,2,\dots,n-1$, $h_1\ge 0$, $h_n\le 0$, and $\sum_{i=1}^{k}{h_i}\ge 0$, $k=1,2,\dots,n-1$, $\sum_{i=1}^{n}{h_i}=0$.\\
%	\indent So for $i<k_1$, there must exist a subsequence $(p,p-1,\dots,1)$ in $H$, for the string $(h_{k_1},h_{k_2},\dots,h_{k_t})$, there must exist a subsequence $(-1,-2,\dots,-q)$.\\
assume that $h_{k_1}=-1$ (resp. $h_{k_t}=-r$) be the first (resp. last) negative entry and $h_j=-q$ be the first $-q$. Then $k_1\le j\le k_t$. Since $|U_{\min}(R)|=1$, $P_R$ consists of precisely one feasible  $\ell$-tuple, say $M$. Since $h_j=-q$, there are at least $q$ positive entries with subscripts less than $k_1$ in $H$.
\begin{claim}\label{CLAIM:c1}
	There are exactly $q$ positive entries $h_{i_1},\ldots,h_{i_q}$ in $H$ satisfying that  $i_1<\ldots<i_q<k_1$.
\end{claim}
\begin{proof}[Proof of Claim~\ref{CLAIM:c1}]
%Since $h_j=-q$, from Theorem 1,Now we prove the number can not be greater than $q$.\\
If not, assume that $j$ occurs in $(i_1,j),\ldots,(i_q,j)$ and there exists an $i_{q+1}<k_1$ such that $h_{i_{q+1}}>0$ and $i_{q+1}$ occurs in $(i_{q+1},j')$ for some $j'\ne j$. Then $k_1\le j'$ and $h_{j'}<0$.
If there exists $i_k$, $k\in \{1,\ldots,q\}$, with $h_{i_k}=1$ then $i_k$ can not occur in any other ordered pairs in $M$. So $M'=(M\setminus\{(i_k,j), (i_{q+1}, j')\})\cup\{(i_{q+1}, j), (i_k, j')\}$ is another feasible $\ell$-tuple of $P_R$, a contradiction to $|P_R|=1$.

  %hence $e_{q+1}\ne e_i$, then we have $e_i<k_1<j$, $e_i<k_1<j'$, $e_{q+1}<k_1<j$, $e_{q+1}<k_1<j'$. So we can exchange $e_i$ and $e_{q+1}$, keep other ordered pairs fixed, we get a different $l$-tuple, this contradicts the uniqueness.\\

 Now assume $h_i>1$ for all $i\in\{i_1,\ldots, i_q\}$ and let $h_{i_0}$ be the last positive entry of $H_1$. Then $h_{i_0}=1$ and $i_0<k_1$. Assume $i_0$ occurs in $(i_0,j'')\in M$.
	
If $k_1\le j''\le k_t$ then $h_{j''}\ge -q$. Since $(i_0,j'')\in M$ and $j''$ occurs precisely $-h_{j''}\le q$ times in $M$, there is at least one  $i\in\{i_1,\ldots,i_q\}$ such that $(i,j'')\notin M$. So $(M\setminus\{(i,j), (i_{0}, j'')\})\cup\{(i_{0}, j), (i, j'')\}$ is a new $\ell$-tuple of $P_R$, a contradiction.

Now assume $j''>k_t$.  Then $h_{j''}<0$. By the maximality of $H_1$ and Fact~\ref{FACT: Landau}, $H$ contains a positive string between $h_{k_t}$ and ${h_{j''}}$, and so  there is $j_0$ such that $k_t<j_0<j''$ and $h_{j_0}=1$. Assume  $(j_0,j''')\in M$. Hence $j_0<j'''$.
%Now we have $i_0<k_t<e_0<j''$, $i_0<k_t<e_0<j'''$, $e_0<j''$, $e_0<j'''$. Then exchange $i_0$ and $e_0$, keep other ordered pairs fixed, we can get another $l$-tuple. Contradiction.
So $(M\setminus\{(i_0,j''), (j_{0}, j''')\})\cup\{(j_{0}, j''), (i_0, j''')\}$ is a feasible $\ell$-tuple different from $M$, a contradiction.
\end{proof}

Furthermore, we have
\begin{claim}\label{CLAIM:c2}
	$h_{i_1},h_{i_2},\ldots,h_{i_q}$ are $q$ distinct positive integers.
\end{claim}
\begin{proof}[Proof of Claim~\ref{CLAIM:c2}] By Claim~\ref{CLAIM:c1}, we have $h_{i_q}=1$.
By Fact~\ref{FACT: Landau}, $\max\{h_{i_1},\ldots,h_{i_q}\}\le q$, and the equality holds if and only if $h_{i_1},\ldots,h_{i_q}$ are pairwise distinct.
So if $h_{i_1},\ldots,h_{i_q}$ are not pairwise distinct then
%	\indent If $h_{i'}=h_{i''}=q$, from the Landau condition, there exists at least $q+1$ positive components whose subscripts are all less that $k_1$.Contradict with claim 1.\\
$\max\{h_{i_1},\ldots,h_{i_q}\}=p<q$. Hence
%	\begin{equation}
$$\sum_{i=i_1}^{i_q}h_i<\sum_{i=1}^{p}i+(p+1)+(p+2)+\dots+q=\sum_{i=1}^{q}i.$$
So
$$	\sum_{i=1}^{k_t}h_i=\sum_{i=i_1}^{i_q}h_i+\sum_{i=k_1}^{k_t}h_i< \sum_{i=1}^{q}i+\sum_{i=1}^{q}(-i)<0, $$
%	\end{equation}
contradicts to the Landau conditions.
\end{proof}

By Claims~\ref{CLAIM:c1},~\ref{CLAIM:c2} and Fact~\ref{FACT: Landau}, $h_{k_1},\ldots, h_{k_t}$ are pairwise different and $(h_{k_1},\ldots, h_{k_t})=(-1,-2,\ldots,-q)$. So we have $H_1=(q, q-1,\ldots, 1,0,\ldots,0,-1,-2,\ldots,-q)$. With the similar discussion on the second maximal segment $H_2$ starting with a positive string  and ending in a nonpositive string with a negative end of $H$, we have $H_2$ is symmetric too. Continuing the same analysis, we get the conclusion.
\end{proof}

\begin{cor}\label{COR: c1}
Let $R=(0,1,\ldots,n-1)$. Then $|U_{\min}(R)|=1$.
\end{cor}
\begin{proof}
Clearly, $H_R=(0,0,\ldots, 0)$ is symmetric. So $|U_{\min}(R)|=1$ by Theorem~\ref{THM: main1}.
\end{proof}

The following examples  show us a score sequence $R$ with $|U_{\min}(R)|=1$ not included in Theorem~\ref{THM: Brau-Li83} and a score sequence $R$ with $|U_{\min}(R)|>1$.
\newtheorem{Example}{\bf{Example}}
\begin{Example}
Let $R=(2,2,2,2,2,5,6,7,9,9,9)$. Then $R$ is feasible and $H_R=(2,1,0,-1,-2,0,0,0,1,0,-1)$. From Theorem~\ref{THM: main1}, $|U_{\min}(R)|=1$ and $$P_R=\{\{(1,5), (2,5), (1,4), (9,11)\}\}.$$
So
%$h_5=-2$ and when $i<5$, only $h_1>0$ and $h_2>0$, so assume that $A$ is one of the tournament matrices whose score sequence are $R$, then there must be $a_{1,5}=a_{2,5}=1$. And $h_4=-1$, so $a_{1,4}=1$ or $a_{2,4}=1$, but $h_2=1$, $a_{2,5}=1$, so $a_{2,j}\ne 1, j\ne 5$, so $a_{1,4}=1$. Moreover, $h_{11}=-1$ and $h_9=1$, the same reason, we have $a_{9,11}=1$. Notice that $h_1+h_2+h_9=4$, then the minimum number of upsets is 4. So $A$ is unique and\\
	\begin{equation}
	A=\left(\begin{array}{ccccccccccc}
    0&0&0&1&1&0&0&0&0&0&0\\
	1&0&0&0&1&0&0&0&0&0&0\\
	1&1&0&0&0&0&0&0&0&0&0\\
	0&1&1&0&0&0&0&0&0&0&0\\
	0&0&1&1&0&0&0&0&0&0&0\\
	1&1&1&1&1&0&0&0&0&0&0\\
	1&1&1&1&1&1&0&0&0&0&0\\
	1&1&1&1&1&1&1&0&0&0&0\\
	1&1&1&1&1&1&1&1&0&0&1\\
	1&1&1&1&1&1&1&1&1&0&0\\
	1&1&1&1&1&1&1&1&0&1&0\\
	\end{array}\right) \notag
	\end{equation}
\end{Example}

\begin{Example}
 Let $R=(2,2,2,2,3,5,6,8,8,8,9)$. Then $R$ is feasible but $H_R=(2,1,0,-1,-1,0,0,1,0,-1,-1)$. Clearly, $H_R$ does not satisfy the requirement of Theorem~\ref{THM: main1}. So $|U_{\min}(R)|>0$. Note that $\ell=2+1+1=4$.
%conditions  still satisfies the Landau condition. We still let $A$ be a tournament matrix whose score sequence is $R$. But now $h_5=-1$, so $a_{1,5}=1$ or $a_{2,5}=1$.\\
%	\indent If $a_{1,5}=1$, then $h_4=-1$, so $a_{1,4}=1$ or $a_{2,4}=1$, if $a_{1,4}=1$, then $a_{2,10}=1$ or $a_{2,11}=1$, if $a_{2,10}=1$, then $a_{8,11}=1$, if $a_{2,11}=1$, then $a_{8,10}=1$; if $a_{2,4}=1$, then $a_{1,10}=1$ or $a_{1,11}=1$, if $a_{1,10}=1$, then $a_{8,11}=1$, if $a_{1,11}=1$, then $a_{8,10}=1$\\
%	\indent If $a_{2,5}=1$, then $a_{1,4}=1$ and $a_{1,10}=1$ or $a_{1,11}=1$, if $a_{1,10}=1$, then $a_{8,11}=1$, if $a_{1,11}=1$, then $a_{8,10}=1$.\\
It can be checked that $P_R$ has six feasible 4-tuples and
\begin{multline*}
  P_R=\{\{(1,4), (1,5), (2,10),(8,11)\}, \{(1,4), (1,5), (2,11),  (8, 10)\}, \\
  \{(1,5), (2,4), (1,10), (8,11)\}, \{(1,5), (2,4), (1,11),(8,10)\},\\
 \{(2,5), (1,4), (1,10),(8,11)\}, \{(2,5), (1,4), (1,11), (8,10)\}\}.
\end{multline*}
%$P_R=\{\{(1,4), (1,5), (2,10),(8,11)\}, \{(1,4), (1,5), (2,11),  \\ (8, 10)\}, \{(1,5), (2,4), (1,10), (8,11)\}, \{(1,5), (2,4), (1,11),(8,10)\}, \{(2,5), (1,4), \\ (1,10),(8,11)\}, \{(2,5), (1,4), (1,11), (8,10)\}\}$.
So
%	\indent So $A$ is not unique and\\
	\begin{equation}
	A=\left(\begin{array}{ccccccccccc}0&0&0&1&1&0&0&0&0&0&0\\
	1&0&0&0&0&0&0&0&0&1&0\\
	1&1&0&0&0&0&0&0&0&0&0\\
	0&1&1&0&0&0&0&0&0&0&0\\
	0&1&1&1&0&0&0&0&0&0&0\\
	1&1&1&1&1&0&0&0&0&0&0\\
	1&1&1&1&1&1&0&0&0&0&0\\
	1&1&1&1&1&1&1&0&0&0&1\\
	1&1&1&1&1&1&1&1&0&0&0\\
	1&0&1&1&1&1&1&1&1&0&0\\
	1&1&1&1&1&1&1&0&1&1&0\\
	\end{array}\right), \notag
	\end{equation}
	
	\begin{equation}
	\left(\begin{array}{ccccccccccc}0&0&0&1&1&0&0&0&0&0&0\\
	1&0&0&0&0&0&0&0&0&0&1\\
	1&1&0&0&0&0&0&0&0&0&0\\
	0&1&1&0&0&0&0&0&0&0&0\\
	0&1&1&1&0&0&0&0&0&0&0\\
	1&1&1&1&1&0&0&0&0&0&0\\
	1&1&1&1&1&1&0&0&0&0&0\\
	1&1&1&1&1&1&1&0&0&1&0\\
	1&1&1&1&1&1&1&1&0&0&0\\
	1&1&1&1&1&1&1&0&1&0&0\\
	1&0&1&1&1&1&1&1&1&1&0\\
	\end{array}\right), \notag
	\end{equation}
	
	\begin{equation}
	\left(\begin{array}{ccccccccccc}0&0&0&0&1&0&0&0&0&1&0\\
	1&0&0&1&0&0&0&0&0&0&0\\
	1&1&0&0&0&0&0&0&0&0&0\\
	1&0&1&0&0&0&0&0&0&0&0\\
	0&1&1&1&0&0&0&0&0&0&0\\
	1&1&1&1&1&0&0&0&0&0&0\\
	1&1&1&1&1&1&0&0&0&0&0\\
	1&1&1&1&1&1&1&0&0&0&1\\
	1&1&1&1&1&1&1&1&0&0&0\\
	0&1&1&1&1&1&1&1&1&0&0\\
	1&1&1&1&1&1&1&0&1&1&0\\
	\end{array}\right), \notag
	\end{equation}
	
	\begin{equation}
	\left(\begin{array}{ccccccccccc}0&0&0&0&1&0&0&0&0&0&1\\
	1&0&0&1&0&0&0&0&0&0&0\\
	1&1&0&0&0&0&0&0&0&0&0\\
	1&0&1&0&0&0&0&0&0&0&0\\
	0&1&1&1&0&0&0&0&0&0&0\\
	1&1&1&1&1&0&0&0&0&0&0\\
	1&1&1&1&1&1&0&0&0&0&0\\
	1&1&1&1&1&1&1&0&0&1&0\\
	1&1&1&1&1&1&1&1&0&0&0\\
	1&1&1&1&1&1&1&0&1&0&0\\
	0&1&1&1&1&1&1&1&1&1&0\\
	\end{array}\right), \notag
	\end{equation}
	
	\begin{equation}
	\left(\begin{array}{ccccccccccc}0&0&0&1&0&0&0&0&0&1&0\\
	1&0&0&0&1&0&0&0&0&0&0\\
	1&1&0&0&0&0&0&0&0&0&0\\
	0&1&1&0&0&0&0&0&0&0&0\\
	1&0&1&1&0&0&0&0&0&0&0\\
	1&1&1&1&1&0&0&0&0&0&0\\
	1&1&1&1&1&1&0&0&0&0&0\\
	1&1&1&1&1&1&1&0&0&0&1\\
	1&1&1&1&1&1&1&1&0&0&0\\
	0&1&1&1&1&1&1&1&1&0&0\\
	1&1&1&1&1&1&1&0&1&1&0\\
	\end{array}\right), \notag
	\end{equation}
	or
	\begin{equation}
	\left(\begin{array}{ccccccccccc}0&0&0&1&0&0&0&0&0&0&1\\
	1&0&0&0&1&0&0&0&0&0&0\\
	1&1&0&0&0&0&0&0&0&0&0\\
	0&1&1&0&0&0&0&0&0&0&0\\
	1&0&1&1&0&0&0&0&0&0&0\\
	1&1&1&1&1&0&0&0&0&0&0\\
	1&1&1&1&1&1&0&0&0&0&0\\
	1&1&1&1&1&1&1&0&0&1&0\\
	1&1&1&1&1&1&1&1&0&0&0\\
	1&1&1&1&1&1&1&0&1&0&0\\
	0&1&1&1&1&1&1&1&1&1&0\\
	\end{array}\right) \notag
	\end{equation}
\end{Example}

\section{Proof of Theorem~\ref{THM: main2}}
\begin{proof}[Proof of Theorem~\ref{THM: main2}]
Let $a_n$ be the number of the feasible score sequence $R=(r_1,r_2,\ldots,r_n)$ with the property that $|U_{\min}(R)|=1$.
Let $H_R=(h_1,h_2,\ldots,h_n)$ be the normalized score vector of $R$. Then $R$ and $H_R$ have a one to one correspondence. Let $b_n$ be the number of the score sequences with $h_1=0$ and $c_n$ be the ones with $h_1\ne 0$. Then $a_n=b_n+c_n$. By Theorem~\ref{THM: main1}, $H_R=(H_1,H_2,\ldots,H_m)$ and every $H_i$ is a symmetric vector for some positive integer $p_i$, $i=1,2,\ldots, m$.
If $h_1=0$ then  $H'=(h_2, \ldots, h_n)$ corresponds to a score sequence $R'=(r_2,\ldots, r_n)$ with $|U_{\min}(R')|=1$.  Thus we have $b_n=a_{n-1}$ for $n\ge 2$ and we define $a_0=1$. To calculate $c_n$, assume that $h_1=p>0$. Then $H_1=(p,p-1, \ldots, 1, 0,\ldots,0,-1,-2,\ldots, -p)$, where the number of zeros between 1 and $-1$ is at least 1 and at most $n-2p$. So $2p+1\le n$. Therefore, the recursion relation of $c_n$ is
%From the results of Theorem 2, when $p$ is small, there could be several `$0$'s between two strings $(p,p-1,\dots,1)$ and $(-1,-2,\dots,-p)$, there also could be several `$0$'s between $H_i$ and $H_j$ from Theorem 2. Hence
\begin{equation}
c_n=\sum_{p=1}^{\lfloor\frac{n-1}{2}\rfloor}\sum_{i=1}^{n-2p}a_{n-2p-i}. \notag
\end{equation}
%where we define $a_0=1$.
So we have
\begin{equation}
a_n=a_{n-1}+\sum_{p=1}^{\lfloor\frac{n-1}{2}\rfloor}\sum_{i=1}^{n-2p}a_{n-2p-i}.
\end{equation}
Let
\begin{equation}
S_n=\sum_{p=1}^{\lfloor\frac{n-1}{2}\rfloor}\sum_{i=1}^{n-2p}a_{n-2p-i}. \notag
\end{equation}
Then $a_n=a_{n-1}+S_n$.
%\begin{equation}
%a_n=a_{n-1}+S_n+\lfloor\frac{n-1}{2}\rfloor \notag
%\end{equation}
When $n$ is odd, it can be  directly checked  that
%\begin{equation}
%S_n=\sum_{p=1}^{\frac{n-1}{2}}\sum_{i=1}^{n-(2p+1)}a_i \notag
%\end{equation}
%\begin{equation}
%S_{n+1}=\sum_{p=1}^{\frac{n-1}{2}}\sum_{i=1}^{n+1-(2p+1)}a_i \notag
%\end{equation}
%\begin{equation}
%S_{n+2}=\sum_{p=1}^{\frac{n+1}{2}}\sum_{i=1}^{n+2-(2p+1)}a_i \notag
%\end{equation}
%\begin{equation}
%S_{n+3}=\sum_{p=1}^{\frac{n+1}{2}}\sum_{i=1}^{n+3-(2p+1)}a_i \notag
%\end{equation}
%So
\begin{equation}
S_{n+1}-S_n=\sum_{p=1}^{\frac{n-1}{2}}a_{n-2p} \notag
\end{equation}
and
\begin{equation}
S_{n+3}-S_{n+2}=\sum_{p=1}^{\frac{n+1}{2}}a_{n+2-2p}. \notag
\end{equation}
So we have
\begin{equation}
S_{n+3}-S_{n+2}-(S_{n+1}-S_n)=a_n. \notag
\end{equation}
When $n$ is even, we similarly have
\begin{equation}
S_{n+3}-S_{n+2}-(S_{n+1}-S_{n})=a_{n}. \notag
\end{equation}
Note that
\begin{equation}
a_{n+1}-a_n=a_n-a_{n-1}+S_{n+1}-S_n \notag
\end{equation}
and
\begin{equation}
a_{n+3}-a_{n+2}=a_{n+2}-a_{n+1}+S_{n+3}-S_{n+2}. \notag
\end{equation}
So
\begin{equation}\label{EQ: e2}
%\begin{aligned}
a_{n+3}-a_{n+2}-(a_{n+1}-a_n)=a_{n+2}-a_{n+1}-(a_n-a_{n-1})+a_n. \notag
%\end{aligned}
\end{equation}
Therefore, the recursion relation of the sequence $\{a_n\}$ is
\begin{equation}
a_{n+3}-2a_{n+2}+a_n-a_{n-1}=0.
\end{equation}
Solve the characteristic equation
\begin{equation}
x^4-2x^3+x-1=0, \notag
\end{equation}
we have
$\lambda_1=\frac{1-\sqrt{2\sqrt{5}+3}}2$, $\lambda_2=\frac{1+\sqrt{2\sqrt{5}+3}}2$, $\lambda_3=\frac{1-i\sqrt{2\sqrt{5}-3}}2$,  $\lambda_4=\frac{1+i\sqrt{2\sqrt{5}-3}}2$.
%\begin{equation}
%\begin{aligned}
%x_1=\frac{1}{2}(1-i\sqrt{2\sqrt{5}-3}),\ x_2=\frac{1}{2}(1+i\sqrt{2\sqrt{5}-3}),\\ x_3=\frac{1}{2}(1-\sqrt{2\sqrt{5}+3}),\ x_4=\frac{1}{2}(1+\sqrt{2\sqrt{5}+3}). \notag
%\end{aligned}
%\end{equation}
So the general formula of $a_n$ is
\begin{equation}
a_n=c_1\lambda_1^{n}+c_2\lambda_2^{n}+c_3\lambda_3^{n}+c_4\lambda_4^{n}. \notag
\end{equation}
It can be easily checked that the original values
\begin{equation}
a_1=1,a_2=1,a_3=2, \mbox{ and } a_4=4.\notag
\end{equation}
So we have
\begin{equation}\label{EQ: e3}
\left\{
\begin{array}{r}
 %c_1+c_2+c_3+c_4 = 1 \\
  c_1\lambda_1+c_2\lambda_2+c_3\lambda_3+c_4\lambda_4 = 1 \\
  c_1\lambda_1^2+c_2\lambda_2^2+c_3\lambda_3^2+c_4\lambda_4^2 = 1  \\
  c_1\lambda_1^3+c_2\lambda_2^3+c_3\lambda_3^3+c_4\lambda_4^3 = 2\\
  c_1\lambda_1^4+c_2\lambda_2^4+c_3\lambda_3^4+c_4\lambda_4^4 = 4
\end{array}
\right.
\end{equation}
or equivalently,
\begin{equation}
\left(\begin{array}{cccc}
1&1&1&1\\
\lambda_1&\lambda_2&\lambda_3&\lambda_4\\
\lambda_1^2&\lambda_2^2&\lambda_3^2&\lambda_4^2\\
\lambda_1^3&\lambda_2^3&\lambda_3^3&\lambda_4^3\\
\end{array}\right)\left(\begin{array}{c}
c_1\lambda_1\\
c_2\lambda_2\\
c_3\lambda_3\\
c_4\lambda_4\\
\end{array}\right)=\left(\begin{array}{c}
1\\
1\\
2\\
4\\
\end{array}\right). \notag
\end{equation}
Solve the system of linear equations, we have
\begin{multline*}
 c_1\lambda_1=\frac{\sqrt{5}+1}4\left(\frac{1}{\sqrt{5}}-\frac{1}{\sqrt{2\sqrt{5}+3}}\right), c_2\lambda_2=\frac{\sqrt{5}+1}4\left(\frac{1}{\sqrt{5}}+\frac{1}{\sqrt{2\sqrt{5}+3}}\right), \\
 c_3\lambda_3=\frac{\sqrt{5}-1}4\left(\frac{1}{\sqrt{5}}-\frac{1}{\sqrt{2\sqrt{5}-3}}i\right), \mbox{ and  } c_4\lambda_4=\frac{\sqrt{5}-1}4\left(\frac{1}{\sqrt{5}}+\frac{1}{\sqrt{2\sqrt{5}-3}}i\right).
\end{multline*}
\end{proof}

\end{document}